\documentclass[letterpaper, 10 pt, conference]{ieeeconf}

\usepackage{amsthm}

\newtheoremstyle{bolden}
  {3pt}
  {3pt}
  {}
  {}
  {\bfseries}
  {:}
  {   }
  {}
\theoremstyle{bolden}

\IEEEoverridecommandlockouts

\usepackage{amsmath, amsfonts}
\usepackage{tikz, graphicx}
\usepackage[section]{placeins}

\newcommand{\E}{\mathbb{E}}

\newcommand{\Lop}{\mathcal{L}}
\newcommand{\V}{\mathcal{V}}
\newcommand{\indic}{\mathbb{I}}
\newcommand{\wdt}{\widetilde}
\newcommand{\sgn}{{\operatorname{sgn}}}

\newcommand{\ep}{\varepsilon}

\newtheorem{theorem}{Theorem}[section]

\newtheorem{corollary}[theorem]{Corollary}
\newtheorem{remark}[theorem]{Remark}

\begin{document}

\title{UAV Circumnavigation of an Unknown Target Without Location Information Using Noisy Range-based Measurements}

\author{
Araz Hashemi $^1$
\thanks{$^1$ Department of Mathematics, Wayne State University, Detroit, MI, 48202, araz.hashemi@gmail.com}
 Yongcan Cao $^2$
\thanks{$^2$ Control Science Center of Excellence, Air Force Research Laboratory, Wright-Patterson AFB, OH 45433 }
 David Casbeer $^2$
George Yin $^{1}$
\thanks{Approved for public release; distribution unlimited, 88ABW-2013-4043.  This work has been supported in part by AFOSR LRIR: 12RB07COR.}
}

\date{}
\maketitle

\begin{abstract}
This paper proposes a control algorithm for a UAV to circumnavigate an unknown
target at a fixed radius when the location information of the UAV is unavailable. By assuming that the UAV has a constant velocity, the control algorithm makes adjustments to the heading angle of the UAV based on range and range rate measurements
from the target, which may be corrupted by additive measurement noise.
The control algorithm has the added benefit of being globally smooth
and bounded. Exploiting the relationship between range rate and bearing angle,
we transform the system dynamics from Cartesian coordinate in terms of location and heading to polar
coordinate in terms of range and bearing angle. We then formulate
the addition of measurement errors as a stochastic differential equation.
A recurrence result is established showing that the UAV will reach a neighborhood of the desired orbit in finite time. Some statistical measures of
performance are obtained to support the technical analysis.
\end{abstract}
\section{Introduction}
Unmanned Aerial Vehicles (UAVs) have been rapidly developing in
capability and hold promise for private, military, and even commercial uses.
From the transport of small goods
in rural areas to the early detection of forest fires \cite{Gertler_US_2012},
UAVs will likely be a ubiquitous tool in coming years. However,
navigation of UAVs is heavily dependent on the use of GPS signals for
location information.
Recent tests show that UAVs are vulnerable to GPS jamming
and spoofing, as evidenced by~\cite{Warwick_Lightsquared_2011},~\cite{Shepard_Dronehack_2012}.
Hence, it is desirable to develop autonomous
control schemes under GPS-denied environment.

A typical application of UAVs is to gather information from a target. In order to obtain enough information regarding a target, it is often necessary to have the UAV orbit around this target at some predetermined distance.
Such a UAV motion is often called \emph{circumnavigation}. While some study has been devoted to the circumnavigation mission,
most control techniques use some type of location information.
In~\cite{Shames_Circumnavigation_2012}, the GPS coordinate of the target is
considered unknown but the location information of the UAV under some local coordinate frame is assumed
to be available. Range measurements from the target are then used to
localize the target; that is, to estimate the relative location of the target from the UAV.
A control algorithm is then designed to produce the desired UAV motion. In~\cite{Deghat_Target_2013}, the dynamics are modeled differently which allows the use of the bearing angle for target localization,
but the location information of the UAV under some local coordinate frame is still assumed.

In~\cite{Cao_Circumnavigation_2013}, Cao et al. exploited a
trigonometric relationship in the system dynamics that allows the
range rate to be used as a proxy for the bearing angle.
It also enables one to transform the UAV dynamics from Cartesian to polar coordinates,
reducing the state space from the 2D location plus the heading angle to simply the range and bearing angle.
Control algorithms were then developed which use range and range rate measurements to
drive the UAV to the desired orbit without the need for target localization
nor the knowledge of the UAV's current position. Clearly, this is
advantageous in situations where GPS is unreliable or unavailable.

In this paper, we expand on the above work to develop a control algorithm for the
circumnavigation task using noisy range and range rate
measurements. In~\cite{Cao_Circumnavigation_2013},
two different control algorithms were developed;
one is smooth but unsaturated, while the other is
saturated but nonsmooth. Both control algorithms were defined only
outside the desired orbit, meaning that zero control input is applied on the inside of the desired orbit to force the UAV to fly straight until it exits again.
To improve the performance we develop a new control algorithm which is
both smooth and saturated via introducing
an appropriate control policy for inside the desired orbit.
In addition, a recurrence result can be established; meaning that the UAV will reach a neighborhood of the desired orbit in finite time, and return if it deviates away from the neighborhood. We then employ numerous examples show the robustness of the new algorithm against measurement noise as well as wind because only range-based measurements are needed.

The rest of the paper is organized as follows.
Section~\ref{sec:formulation} describes the assumed
dynamics and the relations used in the development of the control.
Section~\ref{sec:control} motivates and develops a new control policy based on range and range rate measurements;
first by examining when the UAV
is outside of a given `singular' orbit corresponding to the choice
of one parameter in the control algorithm, and then by examining when the UAV is inside the singular
orbit. Section~\ref{sec:me-analysis} focuses on analyzing the effect of noisy range and range rate measurements on the proposed control algorithm by means of stochastic differential equations (SDEs).
A recurrence result is then established, deriving an upper bound on the time
for the UAV to reach some neighborhood of the desired orbit.
Finally, Section~\ref{sec:simulation} presents a simulation study of
the performance of the control algorithm with noise-corrupted
measurements and collect performance statistics for varying choices of
the gain size. Then the effect of constant
wind is simulated to demonstrate the robustness of the control algorithm when the gain is
appropriately large. Finally, Section~\ref{sec:conclusion} summarizes the
paper and outlines directions for future work.

\section{Probem Formulation} \label{sec:formulation}
The problem set-up is as follows.
Assuming the UAV travels at a constant velocity $V$,
the dynamics are given by
\begin{align} \label{det-cart-dynamics}
  \begin{aligned}
  \dot{x} & = V \cos(\psi) \\
  \dot{y} & = V \sin(\psi) \\
  \dot{\psi} & = u
\end{aligned}
\end{align}
where $[x,y]$ is the 2D location of the UAV, $\psi$ is the heading angle of the UAV, and $u$ is the heading rate to be controlled.
The objective is to design a control algorithm for $u$ such that the UAV orbits some unknown stationary target at a desired radius $r_d$. Considering limited measurements available under GPS-denied environment, the controller has to be constructed based on range measurement $r(t)$ and range rate measurement $\dot{r}(t)$. Here $r(t)$ refers to the distance from the UAV to the target and $\dot{r}$ refers to the rate of $r(t)$.

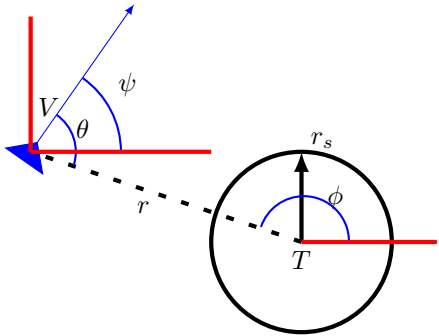
\begin{figure}[h!]
\centering
\begin{tikzpicture}[scale=.6]
\coordinate (center) at (8,3);
\coordinate (uav) at (2,5);
\draw [black, ultra thick] (center) circle [radius = 2.0];
\draw[-latex, black, ultra thick] (center) -- +(90:2);
\node [right] at (8,5.25) {$r_{s}$};
\node [below] at (center) {$T$};
\draw [blue, fill=blue, rotate=-35, shift=(uav)] ++(90:.25) -- ++(-45:.707) -- ++(180:1) -- ++(45:.707);
\draw [black, loosely dashed, ultra thick] (center) -- (uav);
\node [below] at (4.5,4.1) {$r$};
\draw [-latex, blue] (uav) -- ++(55:4);
\node [] at (2.4,6) {$V$};
\draw [blue,thick] (uav) +(55:1) arc [radius=1, start angle=55, end angle=-18.435];
\node [] at (3.15,5.5) {$\theta$};
\draw [blue,thick] (uav) +(55:2) arc [radius=2, start angle=55, end angle=0];
\node [] at (4.15,6.5) {$\psi$};
\draw [blue,thick] (10,2.7) +(161.57:1) arc [radius=1, start angle=0, end angle=161.57];
\node [] at (8.75,4) {$\phi$};
\draw [red, ultra thick] (uav) -- (6,5);
\draw[red, ultra thick] (uav) -- (2,8);
\draw[red, ultra thick] (center) -- (11,3);
\end{tikzpicture}
\caption{Heading angle $\psi$ vs. bearing angle $\theta$ vs. reference angle $\phi$.}
\label{fig:angles}
\end{figure}
For the convenience of notation, we take the target $T$ as the origin of our coordinate frame.
To design a control algorithm and carry out the analysis we shall make use of the
reference angle $\phi$ to the UAV,
as well as the local heading angle $\psi$ of the UAV
and the bearing angle $\theta$ from the reference vector to the heading vector.
See Figure~\ref{fig:angles} for a depiction. We note that
\begin{align} \label{angle-id}
	\begin{aligned}
	\theta = \pi - \phi + \psi.
	\end{aligned}
\end{align}
Then observing
\begin{align}
	\begin{aligned}
	\dot{r} &= \frac{1}{\sqrt{x^2+y^2}} \left[ x \dot{x} + y \dot{y} \right]
		= \cos(\phi) \dot{x} + \sin(\phi) \dot{y},
	\end{aligned}
\end{align}
using the dynamics for $\dot{x}$ and $\dot{y}$ given by \eqref{det-cart-dynamics},
and applying $\phi = \pi - \theta +\psi$
 we arrive at
\begin{align}
	\begin{aligned}
	\dot{r} = -V\cos\theta .
	\end{aligned}
\end{align}
Thus there is a direct correspondence between the bearing
angle $\theta$ and the range rate $\dot{r}$.
This fundamental relation will allow us to use $\dot{r}$ as a
proxy for $\theta$ to design our control.

Also, $\dot{\theta}=-\dot{\phi} + \dot{\psi} =-\dot{\phi}+u$, where
\begin{align}
	\begin{aligned}
	\dot{\phi} &= \frac{ \cos(\phi) \dot{y} - \sin(\phi) \dot{x}}{\sqrt{x^2 +y^2}} = -\frac{V}{r} \sin\theta
	\end{aligned}
\end{align}
so we can transform the system dynamics from $\{(x,y,\phi)\}$ in~\eqref{det-cart-dynamics} to $\{(r,\theta)\}$ given by
\begin{align}
	\begin{aligned}
	\dot{r} &= -V\cos\theta \\
	\dot{\theta} &= \frac{V\sin\theta}{r} + u
	\end{aligned}
\end{align}
The goal is to design a control $u(r,\dot{r}) =u(r, -V\cos\theta)$ such that the dynamics drive
$(r,\theta)$ to $(r_d, \frac{\pi}{2})$.

\section{The Control Algorithm} \label{sec:control}
The designed control algorithm is composed of two cases: (1) $r\geq r_s$; and (2) $r<r_s$, where $r_s<r_a$ is a positive constant defined next. The following two subsections detail how control algorithm is developed for the two cases. 
\subsection{Outer Control}
Suppose that $r\geq r_s$, \textit{i.e.,} the UAV is outside of the black circle as in Figure~\ref{fig:tangent-control}.
The idea for the control algorithm is to drive the UAV towards the tangent point (from the UAV) of the black circle. There is a need
to distinguish between the black circle which is
being aimed for and the `actual' red circle that is achieved, because
we shall see that they are not the same (though an explicit relationship between them
can be identified based on the controller proposed next).
Letting $\gamma = \sin^{-1}\left(\frac{r_s}{r} \right)$, we want to adjust $\psi$ so that $\theta=\gamma$.
Without the ability to measure $\psi$,
it is not possible to make a direct adjustment\footnote{
Note that if we can also measure $\psi$ (e.g. by including a magnetometer to the UAV) in addition to $r$ and $\dot{r}$,
then we can recover coordinates from the identity $\phi = \pi + \psi - \cos^{-1}\left(\frac{-\dot{r}}{V}\right)$ by
\begin{align*} x&= r\cos\phi=r \left[\frac{\dot{r}}{V}\cos\psi - \sin\psi \sin\left(\cos^{-1}\left(\frac{\dot{r}}{V}\right)\right) \sin\psi \right] \\
y&=r\sin\phi=r \left[\frac{-\dot{r}}{V}\sin\psi + \cos\psi \sin\left(\cos^{-1}\left(\frac{\dot{r}}{V}\right)\right) \sin\psi \right]
\end{align*}}.
If $\dot{r}$ is measurable, it can serve as a proxy for $-V\cos\theta$.
Given a preference that the UAV orbit clockwise
(so that $\theta, \gamma \in [0, \pi]$),
$\cos(\cdot)$ is decreasing on $[0,\pi]$. It then can be obtained that
\[
 -(\cos\theta -\cos\gamma) = \cos\gamma-\cos\theta=
 \Big \{
 \begin{array}{l r}
\le 0 &   \theta \ge \gamma \\
\ge 0 &  \theta \le \gamma
 \end{array}
 \]
 and thus
\begin{align}
  \begin{aligned}
    &-k\left[ \dot{r} +V\cos\sin^{-1}\frac{r_s}{r}\right]\\
    &= kV\left[ \cos\theta - \cos\gamma \right]
    = \Big\{
      \begin{array}{lcr}
	<0 & \text{for} &\theta > \gamma \\
	>0 & \text{for} &\theta < \gamma
      \end{array}.
  \end{aligned}
\end{align}
This motivates us to define a control for outside $r_s$ by
\begin{align}
  \label{u_o-defn}
  \begin{aligned}
  u_o(r,\dot{r}) &=-k\left[ \dot{r} +V\cos\sin^{-1}\left( \frac{r_s}{r}
  \right) \right] \indic_{\{r \ge r_s\}} ,\\
  & \qquad \ \text{ or equivalently }  \\
  u_o(r, \theta) &= \left[ kV\cos\theta-kV\frac{\sqrt{r^2-r_s^2}}{r} \right] \indic_{\{r \ge r_s \}}
\end{aligned}
\end{align}
where $k$ is a positive constant. Note that the control is bounded by $2kV$.

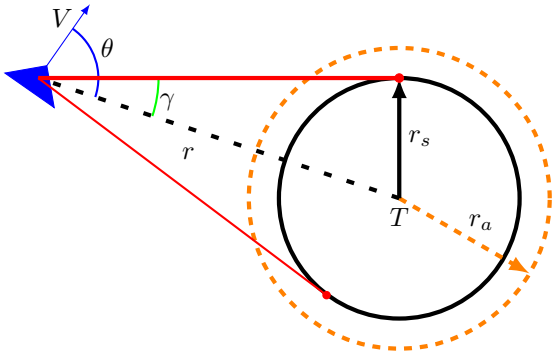
\begin{figure}[h]
\centering
\begin{tikzpicture}[scale=.8]
\coordinate (center) at (8,3);
\coordinate (uav) at (2,5);
\draw [orange, dashed, ultra thick] (center) circle [radius = 2.5];
\draw [black, ultra thick] (center) circle [radius = 2.0];
\draw [-latex, orange, dashed, ultra thick] (center) -- +(-30 : 2.5);
\draw[-latex, black, ultra thick] (center) -- +(90:2);
\node [right] at (9,2.6) {$r_{a}$};
\node [right] at (8,4) {$r_{s}$};
\node [below] at (center) {$T$};
\draw [blue, fill=blue, rotate=-35, shift=(uav)] ++(90:.25) -- ++(-45:.707) -- ++(180:1) -- ++(45:.707);
\draw [black, loosely dashed, ultra thick] (center) -- (uav);
\node [below] at (4.5,4.) {$r$};
\draw [-latex, blue] (uav) -- ++(55:1.5);
\draw [blue,thick] (uav) +(55:1) arc [radius=1, start angle=55, end angle=-18.435];
\draw [green,thick] (uav) +(-18.435:2) arc [radius=2, start angle=-18.435, end angle=0];
\node [] at (4.15,4.6) {$\gamma$};
\node [] at (3.15,5.5) {$\theta$};
\node [] at (2.4,6) {$V$};
\draw [red, ultra thick] (uav) -- (8,5);
\draw [red, fill=red] (8,5) circle [radius=2pt];
\draw [red, thick] (uav) -- +(-37:6) [red, fill=red] circle [radius=1.5pt];
\end{tikzpicture}
\caption{We design a control which aims at the tangent of the orbit of radius $r_s$,
but will `stabilize' at the orbit of radius $r_a$. Here, $\gamma = \sin^{-1}(r_s/r)$.}
\label{fig:tangent-control}
\end{figure}

Interestingly, the UAV cannot stabilize at an orbit of radius $r_s$. Assuming a stable
circular orbit exists with its radius $r_a$, by definition, $\dot{r} = 0$. The nominal angular velocity $ |\frac{V}{r_a}| = |\dot{\psi}| = |u(r_a,0)|$, indicating that
\begin{align}
  \begin{aligned} \label{r_s-defn}
    & \frac{V}{r_a} = kV\cos\sin^{-1} \left( \frac{r_s}{r_a} \right)
    = kV\frac{\sqrt{r_a^2 - r_s^2}}{r_a} \\
    & \implies \frac{1}{k^2} = r_a^2 - r_s^2.
  \end{aligned}
\end{align}
Thus, given any desired actual orbit $r_d$, one may choose a gain size $k \in [\frac{1}{r_d}, \infty)$ and obtain the parameter $r_s = \sqrt{r_d^2 - \frac{1}{k^2}}$ for the control algorithm \eqref{u_o-defn} such that a stable orbit of radius $r_d$ is
feasible. 
From here throughout, we set $r_a=r_d$ so that the actual orbit is equal to the desired orbit, and take $r_s$ as defined by \eqref{r_s-defn}.

\subsection{Inner Control}
When $r<r_s$,~\eqref{u_o-defn} is not well defined due to the term $\cos\sin^{-1}\left(\frac{r_s}{r}\right)$. So a new controller is needed for inside the black circle in Figure~\ref{fig:tangent-control}. In~\cite{Cao_Circumnavigation_2013}, zero control input is applied in order to drive the UAV outside the black circle. One disadvantage of such a control strategy (\textit{i.e.,} zero control for inside the black circle) is that the UAV has to move outside the black circle before control takes affect. As shown in Figure~\ref{fig:u_o-sim-large-gain}, the performance is degraded if the UAV moves inside the black circle quite often. This is particularly true when range and/or range rate measurements are noisy and $r_d$ is close to $r_s$ for large $k$. To keep the UAV from crossing across the desired orbit, similar to the trajectory depicted in Figure~\ref{fig:u_o-sim-small-gain}, a new control algorithm is needed for this case. 


\begin{figure}[ht]
\begin{minipage}{0.45\linewidth}
\centering
\includegraphics[width=\linewidth]{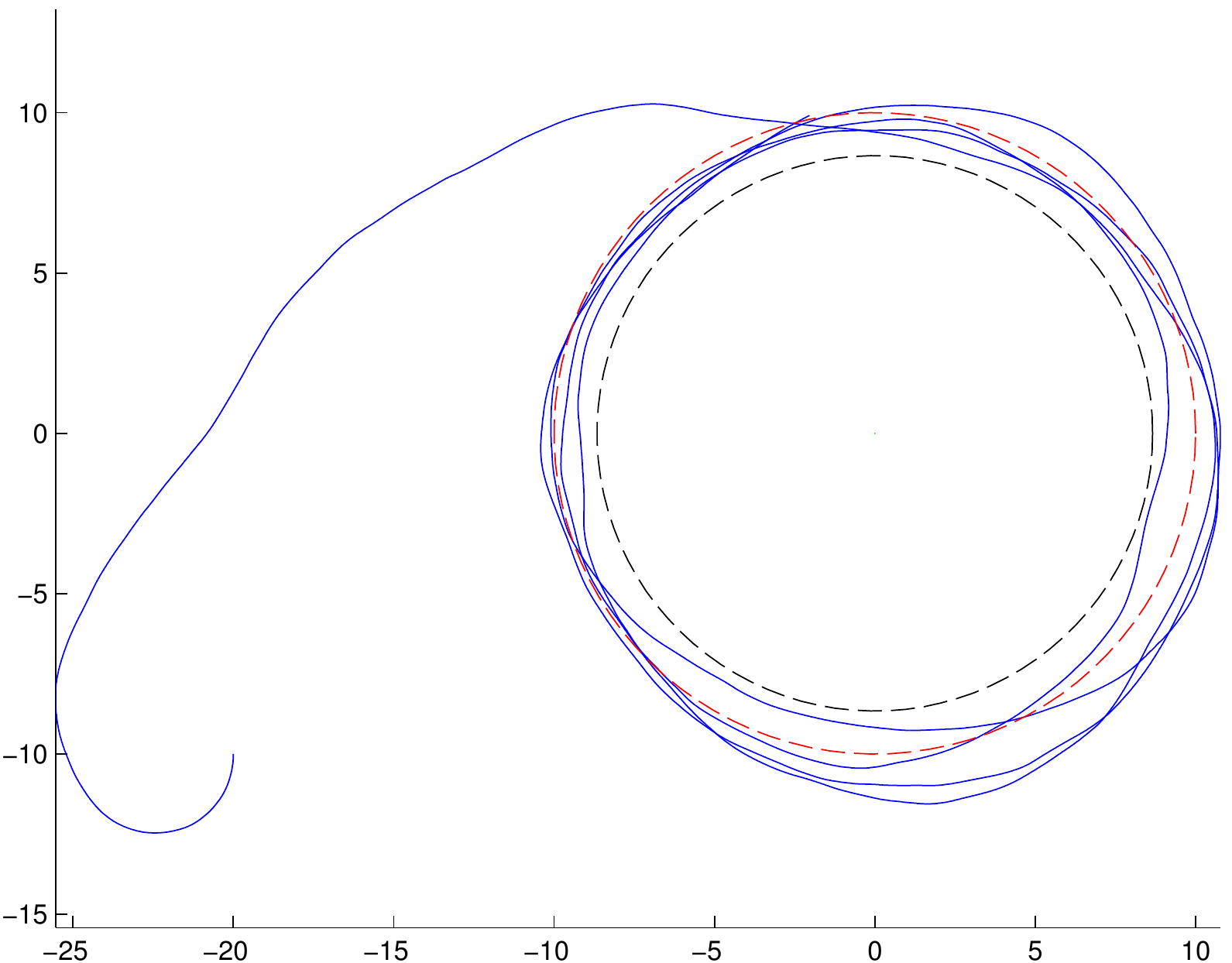}
\caption{A sample trajectory under $u_o$ with small gain: $k=.2$, $r_a=10$, $r_s=8.67$, $V=1$, and additive white measurement noise $\sigma=0.5$.}
\label{fig:u_o-sim-small-gain}
\end{minipage}
\hspace{0.5cm}
\begin{minipage}{0.45\linewidth}
\centering
\includegraphics[width=\linewidth]{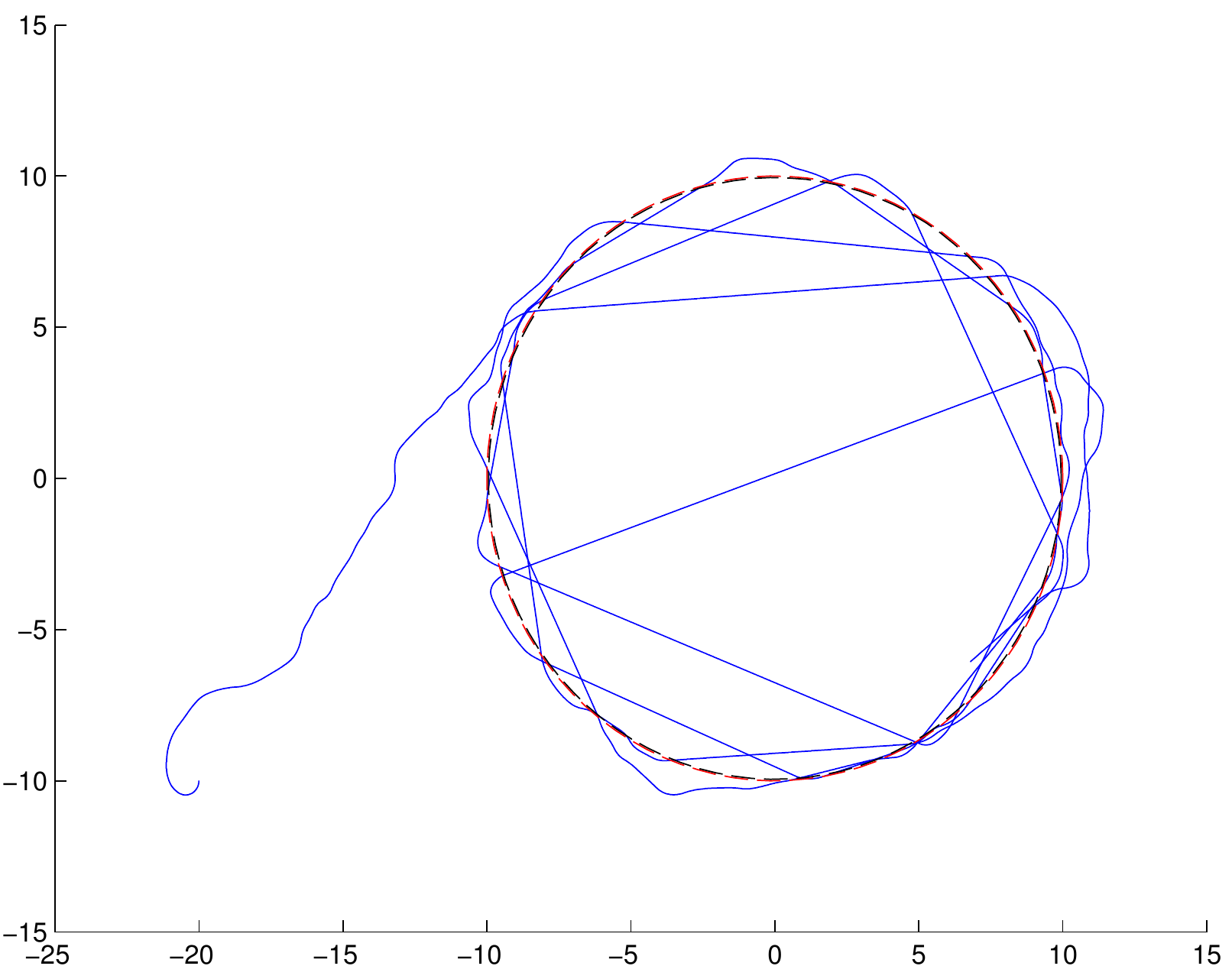}
\caption{A sample trajectory under $u_o$, with large gain: $k=1$
corresponding to $r_s=9.95$. When measurement error nudges the UAV past the
$r_s$ threshold, it cuts across the circle.}
\label{fig:u_o-sim-large-gain}
\end{minipage}
\end{figure}

Note that the two terms in
 $u_o(r, \theta) = kV\cos\theta-kV\frac{\sqrt{r^2-r_s^2}}{r} $
 work separately to adjust the
bearing angle and radius.
If $\theta < \frac{\pi}{2}$ (the bearing is too
acute) then $kV\cos\theta$ is positive and drive the UAV counter clockwise, and does
the reverse if $\theta > \frac{\pi}{2}$. And if $r > r_s$, then
$\frac{-kV}{r} \sqrt{r^2-r_s^2} $
adjusts the heading in such a way that the UAV rotates toward heading the target. This suggests the following inner control as
\begin{align} \label{u_i-defn}
	\begin{aligned}
	u_i(r,\dot{r}) & = -k \left[ \dot{r} - \cos\sin^{-1}\left( \frac{r}{r_s} \right) \right] \indic_{\{r <r_s \}} \\
	u_i(r,\theta) & = \left[ kV\cos\theta + \frac{kV}{r_s}\sqrt{r_s^2-r^2} \right]\indic_{\{r<r_s\}},
	\end{aligned}
\end{align}
where the first component in~\eqref{u_i-defn} is the same as the first component in $u_o$, but
the second component is negated with the nominator and denominator flipped.

Again, a stable orbit of radius $r_i <r_s$ is possible.
If such an orbit exists, it must satisfy $|\frac{V}{r_i}| = |u_i(r_i,0)|$.
By computation, one can obtain
\begin{align} \label{riPM-defn}
  \begin{aligned}
    r_i^2 &=\frac{1}{2} \left[ r_s^2 -\sqrt{r_s^4-\frac{4}{k^2} r_s^2 } \right] \\
    &= \frac{1}{2} \left( r_a^2-\frac{1}{k^2} \right)
    \pm \frac{1}{2} \sqrt{\left( r_a^2 - \frac{1}{k^2} \right)\left(
    r_a^2-\frac{5}{k^2}\right) }.
  \end{aligned}
\end{align}
which has no solution for
$k \in (\frac{1}{r_a}, \frac{\sqrt{5}}{r_a})$,
but otherwise has two solutions $r_{i-} \to 0$
and $r_{i+} \to r_a$ as $k \to \infty$.
These will play some role in the recurrence analysis.
\begin{remark}\label{rem:r_i-instability}
We note that the UAV can only stabilize at one of the inner stable
radii $r_i$ if the initial point and heading is exactly along the orbit with radius $r_i$ in a
counter-clockwise orientation,
corresponding to $(r(0),\theta(0))=(r_i,3\pi/2)$, thus forcing the `$\theta$'
(or $\dot{r}$)
component of the control $kV\cos\theta$ in \eqref{u_i-defn} to be 0.
However, any perturbation of the inputs for the control
which force the UAV even negligibly off-course will cause the $\theta$
component to drive the UAV's bearing angle towards $\pi/2$ because $(r_i,3\pi/2)$ is an unstable equilibrium. 
Eventually, the UAV will be driven outside the orbit with radius $r_s$. In the
presence of measurement errors, the UAV
is driven outside the orbit with radius $r_s$ almost immediately as evidenced by
Figure~\ref{fig:u-sim-in}. Other simulations demonstrate that even if $(r(0),\theta(0))=(r_i, 3\pi/2)$ and no measurement errors exist,
accumulated numerical errors
will eventually drive the UAV slightly off the orbit of radius $r_i$ after which
it immediately moves outside the orbit with radius $r_s$.
Hence the inner stable orbits
are of little practical concern for the implementation
 of the control algorithm.
\end{remark}

As a summarization, the proposed control algorithm is given by $u = u_o + u_i$; that is
\begin{align} \label{u-defn}
	\begin{aligned}
	u(r, \dot{r}) =& -k\dot{r}
	-kV \cos\sin^{-1} \left( \frac{r_s}{r} \right) \indic_{\{r > r_s\}}\\
	& 
	+kV \cos\sin^{-1} \left( \frac{r}{r_s} \right) \indic_{\{r < r_s\}} 
	\end{aligned}
\end{align}
or equivalently 
\begin{align*}
	\begin{aligned}
	u(r,\theta) =& kV\cos\theta -\frac{kV}{r}\sqrt{r^2-r_s^2}\indic_{r>r_s}\\
	&+\frac{kV}{r_s}\sqrt{r_s^2-r^2} \indic_{\{r<r_s\}}.
	\end{aligned}
\end{align*}
As an example, Figures~\ref{fig:u-sim-out} and~\ref{fig:u-sim-in} depict the improved performance of the UAV under the proposed control algorithm~\eqref{u-defn} with $k=1$. Notice that the UAV will eventually stay close to the desired orbit as opposed to the behavior seen in Figure~\ref{fig:u_o-sim-large-gain} when zero control is applied for the case $r(t)<r_s$.

\begin{figure}[ht]
\begin{minipage}[b]{0.45\linewidth}
\centering
\includegraphics[width=\textwidth]{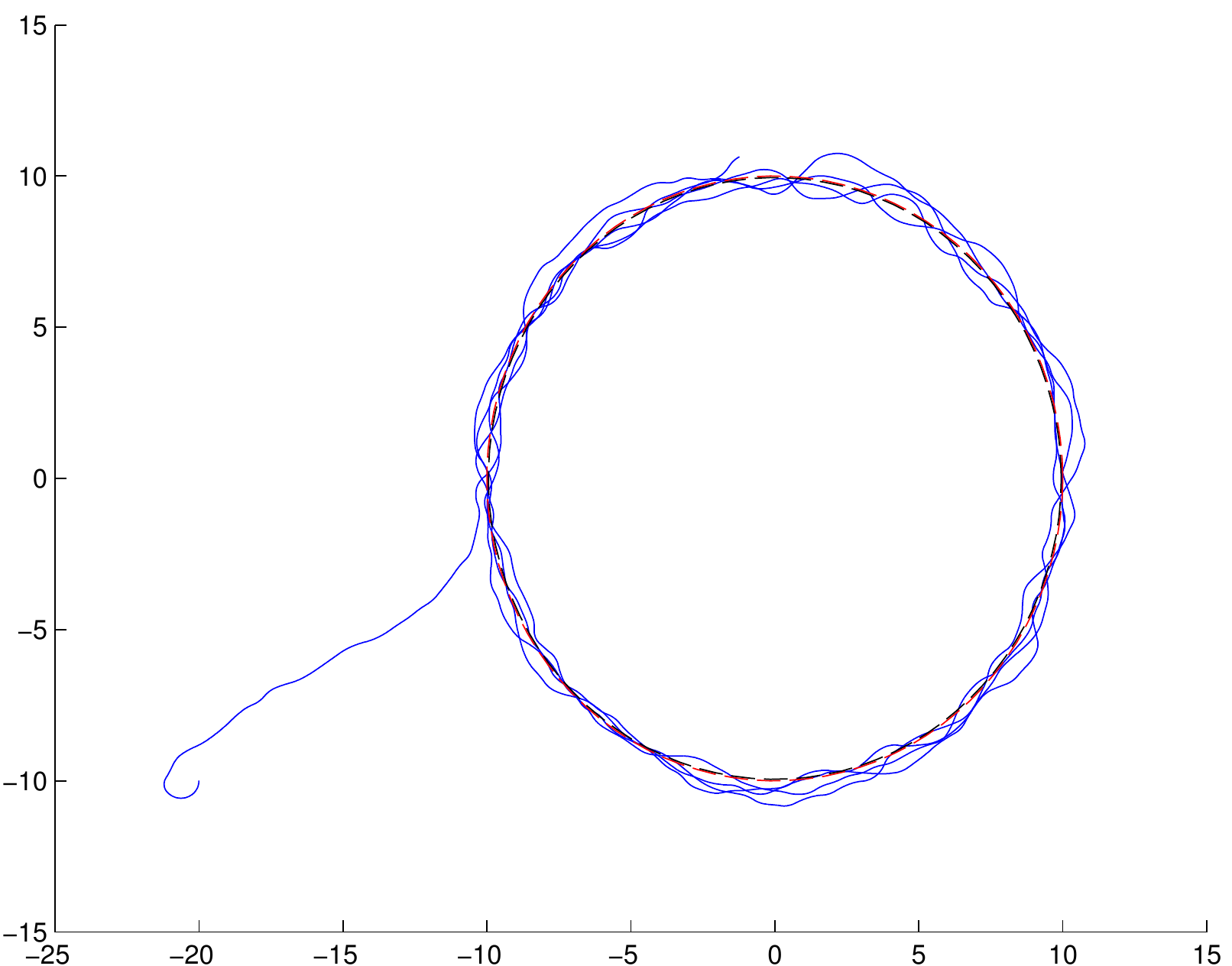}
\caption{A sample trajectory under $u$ with initial point outside the desired orbit}
\label{fig:u-sim-out}
\end{minipage}
\hspace{0.5cm}
\begin{minipage}[b]{0.45\linewidth}
\centering
\includegraphics[width=\textwidth]{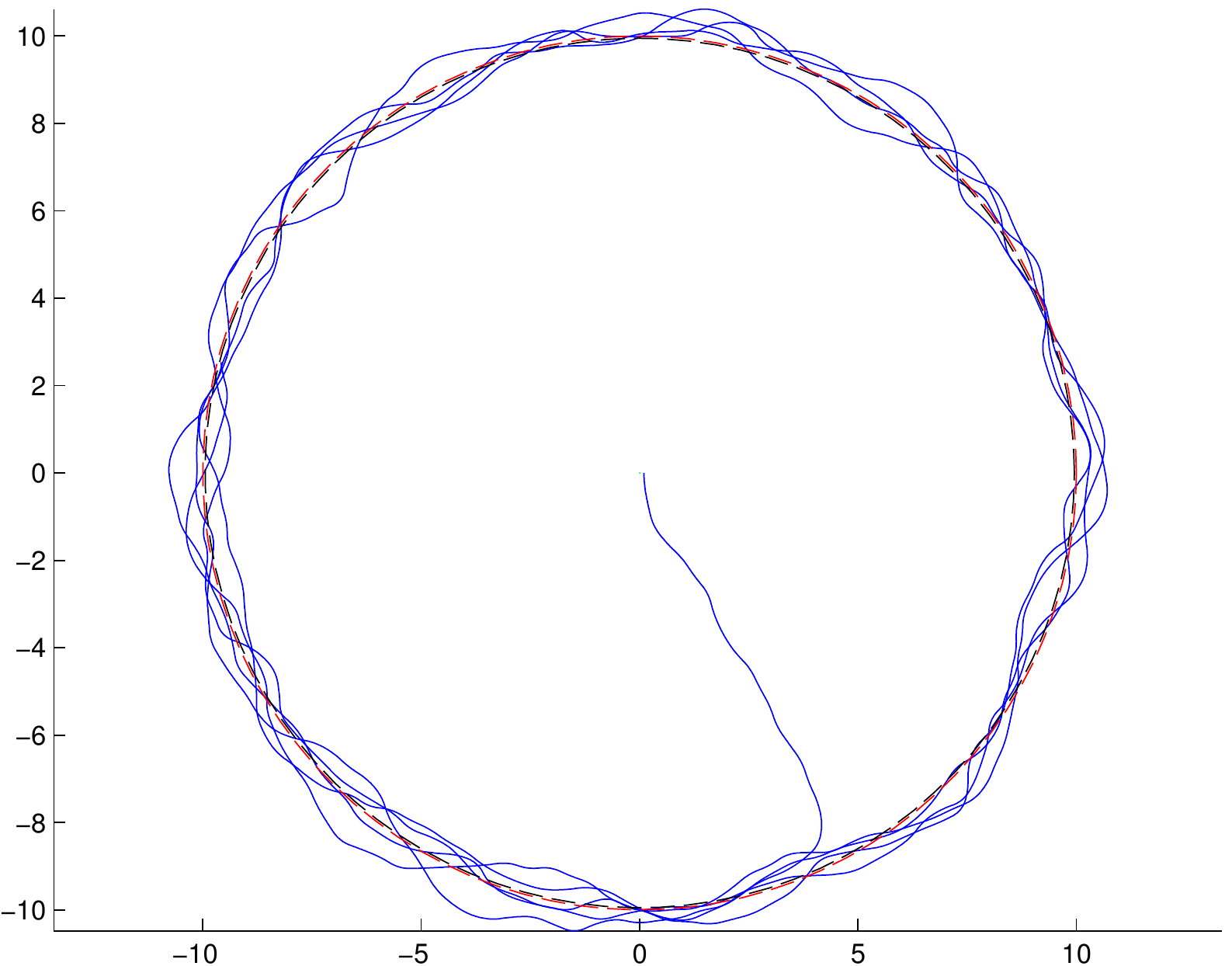}
\caption{Sample trajectory under $u$ with initial point inside the desired orbit}
\label{fig:u-sim-in}
\end{minipage}
\end{figure}

\section{Measurement Error Analysis} \label{sec:me-analysis}

\subsection{SDE Formulation}

Here we formally introduce additive measurement noises in the controller.
For example, range $r$ can be measured accurately,
but range rate measurement is noisy
$\wdt{\dot{r}} = \dot{r} + \nu$ where
$\nu \sim \mathcal{N} (0, \sigma) $.
This model has practicality, as the range measurements are
tremendously accurate compared to range rate measurements
regardless of what method we use for the estimation.
Then the noisy control input becomes
\begin{align*}
  \wdt{u}(r, \theta, \nu) &\overset{\Delta}{=} u(r, \dot{r} + \nu)
   \overset{\Delta}{=} u(r, \theta) -k\nu.
\end{align*}

With the noisy control input, the noisy system dynamics are modeled by the
stochastic differential equation
\begin{align}
\begin{aligned}
d\left[\begin{array}{c}
r\\ \theta
\end{array}\right]
=\left[\begin{array}{c}
-V\cos\theta\\
\frac{V\sin\theta}{r}+u(r,\theta)
\end{array}\right]dt+\left[\begin{array}{c}
0\\-k\sigma
\end{array}\right]
d\xi
\end{aligned}
\label{me-sde}
\end{align}
where $\xi$ is a standard Brownian motion.
One can verify that
 the control defined by \eqref{u-defn}
has linear growth and is Lipschitz continuous
(even at $r=r_s$), and the other coefficients also satisfy this property
on domains bounded away from $r=0$.
Hence \eqref{me-sde} describes an Ito diffusion,
and thus a unique Markov solution exists for the trajectory as in
\cite[Definition 7.1.1, Theorem 5.2.1]{Oksendal_Stochastic_2003}.
The associated generator $\Lop$ of the diffusion is given by
\begin{align}
\begin{aligned}&\Lop\V(r,\theta)  =\left[-V\cos\theta\right]\frac{\partial}{\partial r}\V(r,\theta)\\
& +\left[\frac{V\sin\theta}{r}+u(r,\theta)\right]\frac{\partial}{\partial\theta}\V(r,\theta)+\frac{k^{2}\sigma_{rr}^{2}}{2}\frac{\partial^{2}}{\partial\theta^{2}}\V(r,\theta).\end{aligned}
\label{Lop-defn}
\end{align}

\subsection{A Recurrence Result}
Let $Z(t)$ be an $\ell$-dimensional diffusion process. It is said to be regular if
it does not blow up in finite time w.p.1.  Suppose that
$Z(t)$ is an $\ell$-dimensional diffusion process that is regular, that $D$ is
an open set with compact closure, that  $Z(0)=z\in D^c$ the complement of $D$,
and that $\sigma^{z}_{D}=
\inf\{t: Z^{z}(t)\in D\},$
where $Z^z(t)$ signifies the initial data $z$ dependence of the diffusion. The process
$Z^{z}(\cdot)$ is {\em recurrent} with respect
to $D$ if $ P(\sigma^{z}_{D}<\infty)=1$
for any $ z \in D^c$; otherwise,
the process is {\em transient} with respect to $D$.
A recurrent
process with finite mean recurrence time for some set $D$
is said to be {\em positive
recurrent} w.r.t. $D$;
otherwise, the process is {\em null recurrent}
w.r.t. $D$.

Coming back to our problem,
we shall show that the trajectory of the UAV under
control policy \eqref{u-defn} with dynamics given by
\eqref{me-sde} is recurrent with respect
 to a neighborhood of either
$r=r_a$ or $r=0$, as depicted in Figure~\ref{fig:recurrent-set}.
The recurrence is in the sense that if the initial point of the UAV is
outside of the recurrent set, the UAV will enter the recurrent set in
finite time almost surely.

\begin{figure}[h]
\centering
\begin{tikzpicture}[scale=.55]
\begin{scope}[shift = {(5,5)}]
\coordinate (center) at (0, 0);
\draw[fill=green](center) circle [radius=5.15];
\draw[fill=white ](center) circle [radius=4.5];
\draw[ultra thick, violet](center) circle [radius=5.15];
\draw[ultra thick, blue](center) circle [radius=4.5];
\draw [red, dashed, ultra thick] (center) circle [radius = 5];
\draw [black, ultra thick, dashed] (center) circle [radius = 4.75];
\draw[fill=green](center) circle [radius=.75];
\draw[ultra thick, blue](center) circle [radius=.75];
\draw[red, ultra thick](center) -- (-5,0);
\node[above, red] at (-3.5,0) {$r_a$};
\draw[black, ultra thick](center) -- (-4.6435, -1);
\node[right, black] at (-3, -1) {$r_s$};
\draw[violet, ultra thick] (center) -- (-4.1860,3);
\node[violet, left] at (-4.1860,3) {$r_{a,\varepsilon}$};
\draw[blue, ultra thick] (center) -- (-3.3541, -3);
\node[blue, right] at (-2.8, -2.75) {$r_{i+,\varepsilon}$};
\draw[blue, ultra thick] (center) -- ( 0, -.75);
\node[blue, below] at ( 0, -.8) {$ r_{i-, \varepsilon} $};
\draw [blue, fill=blue, shift={(0,5)}] ++(0: 0.5) -- ++(135: 0.7071) -- ++(-90:1) -- ++(45:.7071);
\draw [blue, fill=blue, rotate=0, shift={(.75,0)}] ++(90:.25) -- ++(-45:.707) -- ++(180:1) -- ++(45:.707);
\end{scope}
\end{tikzpicture}
\caption{The recurrent set $U_{k, \ep}$. }
\label{fig:recurrent-set}
\end{figure}
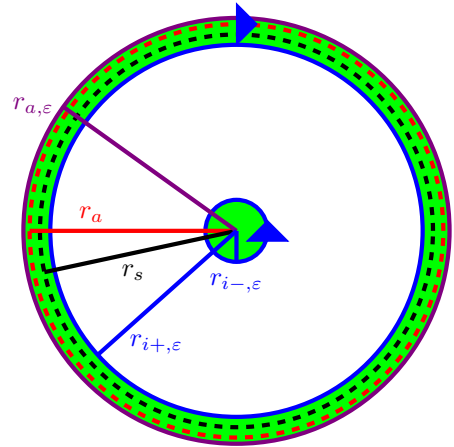

We shall prove our result using a Lyapunov function approach.
Consider the candidate function
\begin{align}\label{sgn-lya-fcn}
\V(r,\theta) = \frac{k}{V}  |r-r_s| + \frac{\theta}{V} \sgn(r-r_s)
+\frac{2\pi}{V}
\end{align}
which is everywhere positive on the domain
 $r \in (0,r_s) \cup (r_s, \infty)$ and $\theta \in [0, 2\pi)$.
Note that $\V$ by \eqref{sgn-lya-fcn} is not differentiable along
 $r=r_s$. However, this will become part of the recurrent set $U_{k, \ep}$
 and it is only on the complement set $U_{k, \ep}^c$ which the Lyapunov
 function must be smooth.
On such a domain, we have that
\begin{align}
	\begin{aligned} \label{LopV}
	\Lop\V &= -k\cos\theta\sgn(r-r_s)
	+\frac{\sin\theta}{r}\sgn(r-r_s) \\
	&  \qquad+k\cos\theta\sgn(r-r_s)
	+u(r)\sgn(r-r_s) \\
	&=\frac{\sin\theta}{r}\sgn(r-r_s)
	-\frac{k}{r}\sqrt{r^2-r_s^2}\indic_{\{r>r_s\}} \\
	& \qquad \qquad \qquad
	-\frac{k}{r_s}\sqrt{r_s^2-r^2}\indic_{\{r<r_s\}}.
	\end{aligned}
\end{align}

\begin{theorem}\label{thm:r_ep-exist}
For $\ep$ sufficiently small and $k$ sufficiently large,
there exists
\begin{align}
 \begin{aligned}
 &r_{i-,\ep} \searrow r_{i-} \quad
 & r_{i+,\ep} \nearrow r_{i+} \qquad
 & r_{a,\ep} \searrow r_a \quad
 & \text{ as } \ep \downarrow 0 \\
 &\text{ where } & & & \\
 &r_{i-} \searrow 0 \quad
& r_{i+} \nearrow r_s \qquad
 & r_s \nearrow r_a \quad
 & \text{ as } k \uparrow
\end{aligned}
\end{align}
such that $\Lop \V \le -\ep$ on $U_{k, \ep}^c$, where
\begin{align} \label{U_ep-defn}
\begin{aligned}
U_{k, \ep} \overset{\Delta}{=}  &
 \{(0,r_{i-,\ep}) \times (\pi,2\pi)\}
  \cup \{ (r_{i+,\ep}, r_{a,\ep}) \times (0, \pi) \}.
 \end{aligned}
 \end{align}
\end{theorem}

With the above, using \cite[Theorem 3.9]{Khasminskii_Stochastic_2011},
we can obtain the following corollary.

\begin{corollary}[Recurrence Time Bound] \label{cor:rec-time-bound}
 For $\ep$ sufficiently small and $k$ sufficiently large,
 the trajectory of the UAV derived from \eqref{me-sde}
 under control policy \eqref{u-defn} is recurrent to $U_{k, \ep}$
 as defined in \eqref{U_ep-defn}. Given an initial point
 $(r_0, \theta_0)$, the expected recurrence time $\tau_\ep$
 until the UAV reaches $U_{k, \ep}$ is bounded by
 \begin{align}
 \E^{(r_0, \theta_0)} \tau_\ep \le
\frac{\V(r_0,\theta_0)}{\ep}
=  \frac{k|r_0-r_s| +\theta_0 +2\pi}{V \ep}.
\end{align}
\end{corollary}

\begin{proof}[Proof of Theorem~\ref{thm:r_ep-exist}]
We see that the second and third terms of \eqref{LopV} are always
non-positive. If $r>r_s$ and $\theta \in (\pi, 2\pi)$ then $\Lop \V < 0$.
Similarly if $r<r_s$ and $\theta \in (0, \pi)$, then $\Lop\V <0$.

We note that $\Lop\V \le 0$ for $r \ge r_a$, regardless of $\theta$.
In particular, considering the worst case scenario $\sin\theta =1$ we can
solve for $r>r_s$
  such that
  \begin{align*}
    \begin{aligned}
      \Lop \V (r) &= \frac{1}{r} \left[ 1 - k\sqrt{r^2-r_s^2} \right]
      \le -\ep.
    \end{aligned}
  \end{align*}
  This has a solution if $\ep \le k$ (where $k$ can be taken in
  $[\frac{1}{r_a}, \infty)$) and leads us to define
    \begin{align}
      \begin{aligned}
	r_{a,\ep} & \overset{\Delta}{=}
	\frac{\ep + \sqrt{ k^2 r_a^2[k^2-\ep^2] + \ep^2} }
	{k^2 -\ep^2} .
      \end{aligned}
    \end{align}
Then $\Lop \V (r,\theta) \le -\ep$ for $r \ge r_{a,\ep}$ regardless of $\theta$.
As $\ep \downarrow 0$ or as $k \uparrow \infty$, we have $r_{a,\ep}
\downarrow r_a$. Thus we can force $r_{a,\ep}$ arbitrarily close to $r_a$.

If $r<r_s$, then
\begin{align*}
	\Lop \V = \frac{-\sin\theta}{r} - \frac{k}{r_s}\sqrt{r_s^2-r^2}.
\end{align*}
Again considering the worst-case scenario $\sin\theta=-1$,
we inspect the function
\begin{align}
	g(r) = \frac{1}{r} - \frac{k}{r_s}\sqrt{r_s^2-r^2}
\end{align}
and solve for $r_i$ such that $g(r_i)=0$.
This reduces to \eqref{riPM-defn},
which has no solutions in $(0, r_s)$ for
$k \in (\frac{1}{r_a}, \frac{\sqrt{5}}{r_a})$,
but otherwise has two solutions $r_{i-} \to 0$
and $r_{i+} \to r_a$ as $k \to \infty$.
If $r_{i-} \le r \le r_{i+}$, then $\Lop \V \le 0$.
If $k<\frac{\sqrt{5}}{r_a}$,
then $\Lop\V$ is always positive in a neighborhood of
$\theta=3\pi/2$ for all $0< r \le r_s$.

Repeating the process to solve where $g(r) =-\ep$,
we obtain the quartic equation
\begin{align}\label{riPM_ep-defn}
	r^4 + \frac{r_s^2}{k^2} \left( \ep^2 - k^2 \right) r^2
	- 2\ep \frac{r_s^2}{k^2} r + \frac{r_s^2}{k^2}=0
\end{align}
which has two solutions  $r_{i-,\ep} $ and $r_{i+,\ep}$ in
$(r_{i-}, r_{i+})$ for sufficiently small $\ep$.
Between $r_{i-,\ep}$ and $r_{i+, \ep}$ we have that $g(r) \le -\ep$ ,
with $r_{i-,\ep} \downarrow r_{i-}$ and $r_{i+,\ep} \uparrow r_{i+}$
as $\ep \downarrow 0$. Then using $r_{i-} \searrow 0$,
$r_{i+} \nearrow r_s$, and $r_s \nearrow r_a$ as $k \uparrow$,
the corollary stands.
\end{proof}

\begin{remark}[$\ep$ Upper Bound]
We note that the upper bound on the recurrence time $\tau_\ep$ given 
in Corollary~\ref{cor:rec-time-bound} is inversely proportional to $\ep$ 
(corresponding to the size of the recurrent set $U_{k,\ep}$). 
Thus allowing for a larger neighborhood of our desired orbit will decrease the bound for the time $\tau_\ep$ it takes to reach said neighborhood. One may wonder how large we may take $\ep$ to be while 
still being able to solve for a recurrent set $U_{k,\ep}$, off of which 
$\Lop \V \le -\ep$.
To find the maximum value of $\ep$ which allows for the result,
one may analyze the function
$g(r)=\frac{1}{r} -\frac{k}{s} \sqrt{s^2-r^2}$, where
$s=\sqrt{r_a^2-k^{-2}}$ varies with $k$ but is bounded between
$0$ and $r_a$.
Heuristically, one sees that the minimum value of $g(r)$ is $-O(k)$,
and thus the maximum possible value of $\ep$ is $O(k)$.
To find the explicit bound, one finds
\begin{align*}
  g'(r) &= \frac{kr}{s \sqrt{s^2-r^2}} -\frac{1}{r^2}=0
  \implies r_*^6 +\frac{s^2}{k^2}r_*2 -\frac{s^4}{k^2} = 0
\end{align*}
which has a unique real solution $r_*$ in $(r_{i-}, r_{i+})$ given by
\begin{align}
	\begin{aligned}
  r_*^2 &= \sqrt[3]{ \frac{9(sk)^4 + \sqrt{81(sk)^8+12(sk)^6}}{18k^6} }\\
  & \quad - \sqrt[3]{ \frac{\frac{2}{3}s^6}{9(sk)^4 + \sqrt{81(sk)^8+12(sk)^6}} }
  	\end{aligned}
\end{align}
whose evaluation in $g(r_*)$ gives the lower bound needed for the
analysis inside $r<r_s$.
Thus taking $\ep < \min\{g(r_*), r_a^{-1} \}$
will yield a valid result.
\end{remark}

\begin{remark}[$k$ `Practical' Upper Bound]
For a fixed value of $k$, one may let $\ep \searrow 0$ and obtain
a `minimal' recurrent set
\[U_k = \{(0,r_{i-}) \times (\pi, 3\pi/2) \}
\cup \{ (r_{i+}, r_a) \times (0,\pi) \}. \]
While analytically one may take $k$ arbitrarily large to force
$r_{i-} \searrow 0$ and $r_{i+} \nearrow r_s \nearrow r_a$ and tighten
the minimal recurrent set, practically one encounters problems if the gain
is too large. If the maximum control effort $2kV$ is larger than $\pi$, then
(in addition to clearly violating practical turning constraints)
it is possible for the UAV to spin out, resulting in significant deviations
from the desired orbit. We shall observe this in the simulation study, e.g.,
Figure~\ref{fig:noisy-r_erravg2}.
\end{remark}

\section{Simulation Study}\label{sec:simulation}

\subsection{Measurement Error, Windless}
Here we simulate the performance of the control algorithm \eqref{u-defn}
with additive measurement errors in the absence of wind, as in \eqref{me-sde}.
The desired orbit is of radius $r_a=10$. We take the velocity of the UAV
$V=1$ and the standard deviation of the
measurement error $\sigma = 0.5$.
We run the simulation for 350 seconds, updating the control every 0.5
seconds.
\begin{figure}[ht]
\begin{minipage}{0.49\linewidth}
\centering
\includegraphics[width=\linewidth]{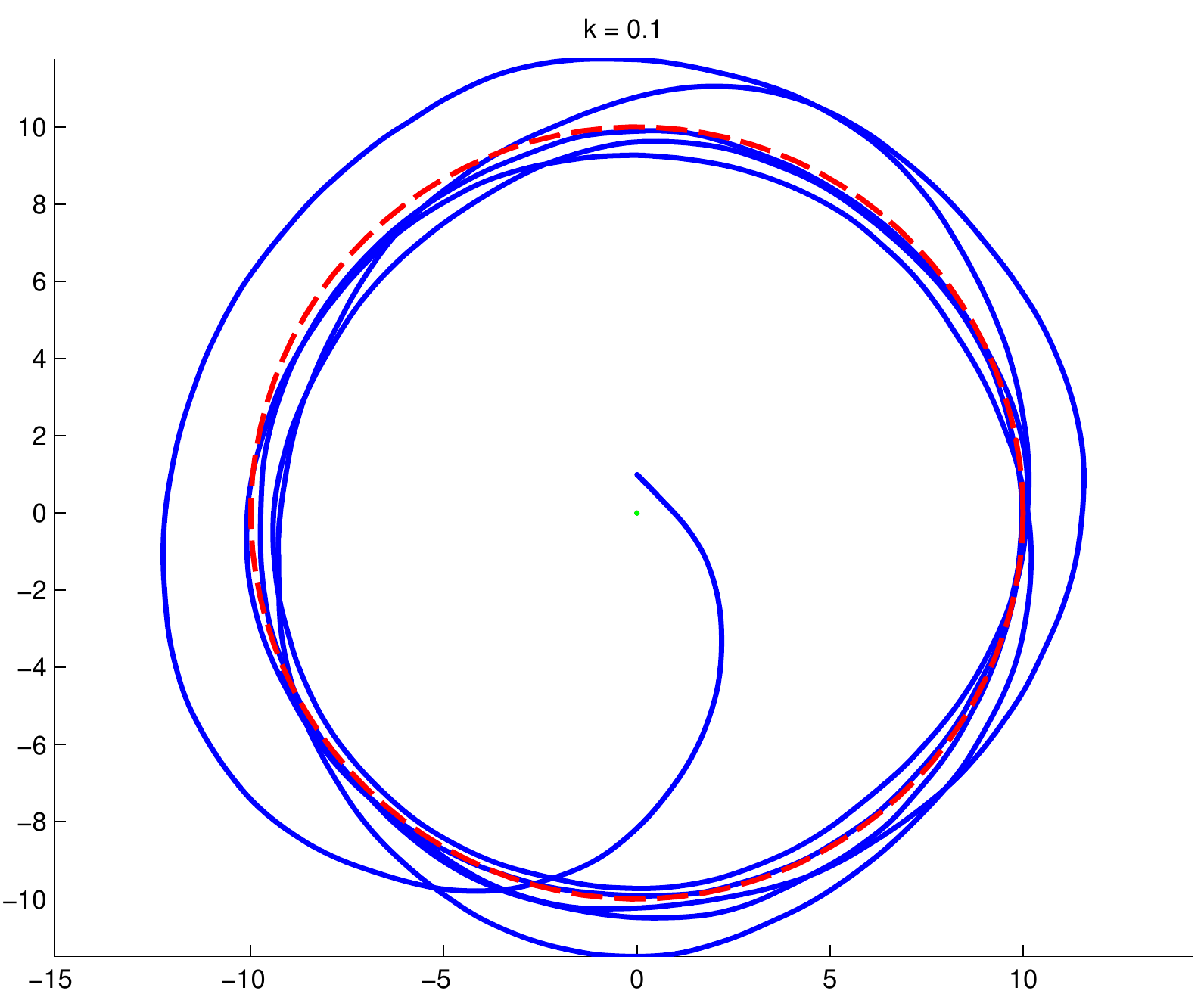}
\caption{Trajectory with measurement error, $k=0.1$}
\label{fig:noisy-k_0_1}
\end{minipage}
\begin{minipage}{0.49\linewidth}
\centering
\includegraphics[width=\linewidth]{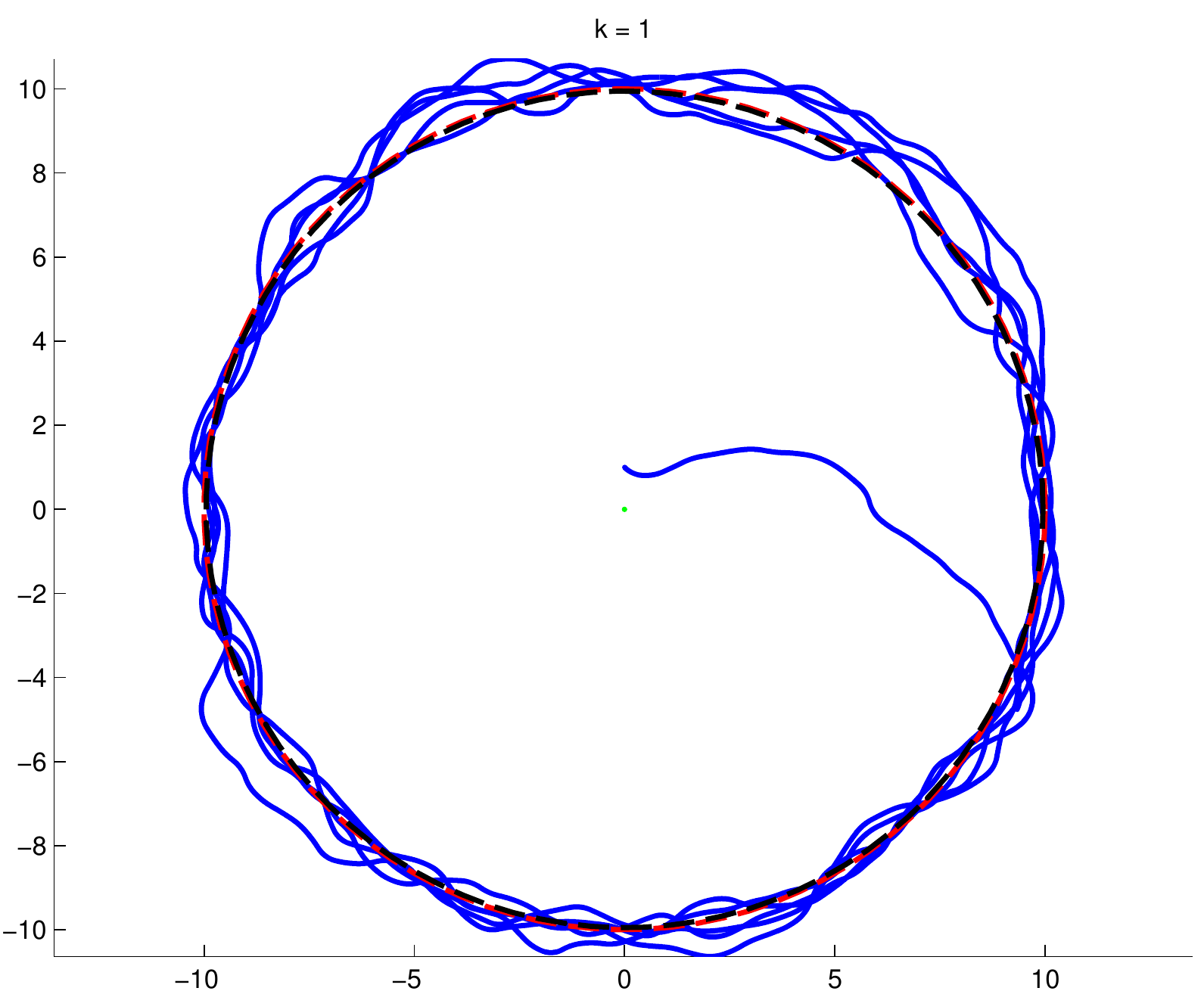}
\caption{Trajectory with measurement error, $k=1$}
\label{fig:noisy-k_1_0}
\end{minipage}
\end{figure}
Figure~\ref{fig:noisy-k_0_1} shows
the trajectory of the of UAV with gain size $k=0.1 = r_a^{-1}$
corresponding to $r_s=0$, while Figure~\ref{fig:noisy-k_1_0} shows
the trajectory with gain size $k=1.0$ corresponding to $r_s=9.95$.
We observe that the smaller gain size gives a smoother trajectory
but larger deviations from the desired radius. The larger gain size
adheres to the desired orbit more closely, but at the expense of a
larger control effort.

We then run the simulation 20 times, increasing the gain $k$ on each
iteration from the minimum value $k=0.1$ by increments of $0.15$, and
collect statistics its performance. Figures~\ref{fig:noisy-r_erravg2} and
\ref{fig:noisy-rdot_erravg2} show the
average of $(r-r_a)^2$ and $\dot{r}^2$ respectively as the gain $k$
increases. This supports the observation from the trajectories that higher
gain choices correspond to less radial error at the expense of
smoothness and large control effort; though only to a point. If the maximum
control adjustment $2kV$ is larger than $\pi$ (here corresponding when
$k=\pi/2$), then the UAV may turn directly around instantaneously.
Besides being quite impractical, this causes the UAV to over-correct and
spin out of control.

\begin{figure}[h]
\centering
\includegraphics[width=0.75\linewidth]{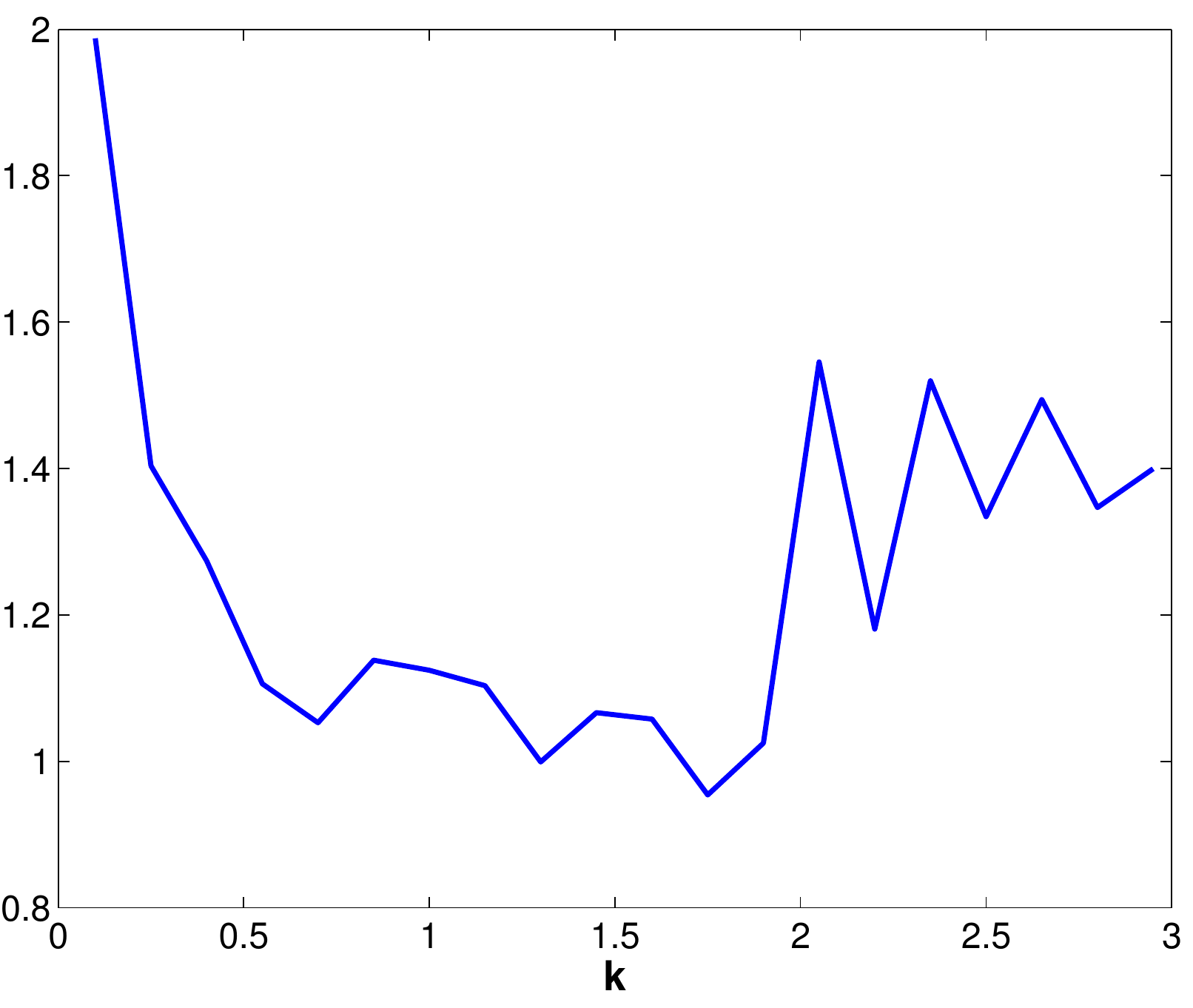}
\caption{Average mean-square error of $(r-r_a)$}
\label{fig:noisy-r_erravg2}
\end{figure}
\begin{figure}[h]
\centering
\includegraphics[width=0.75\linewidth]{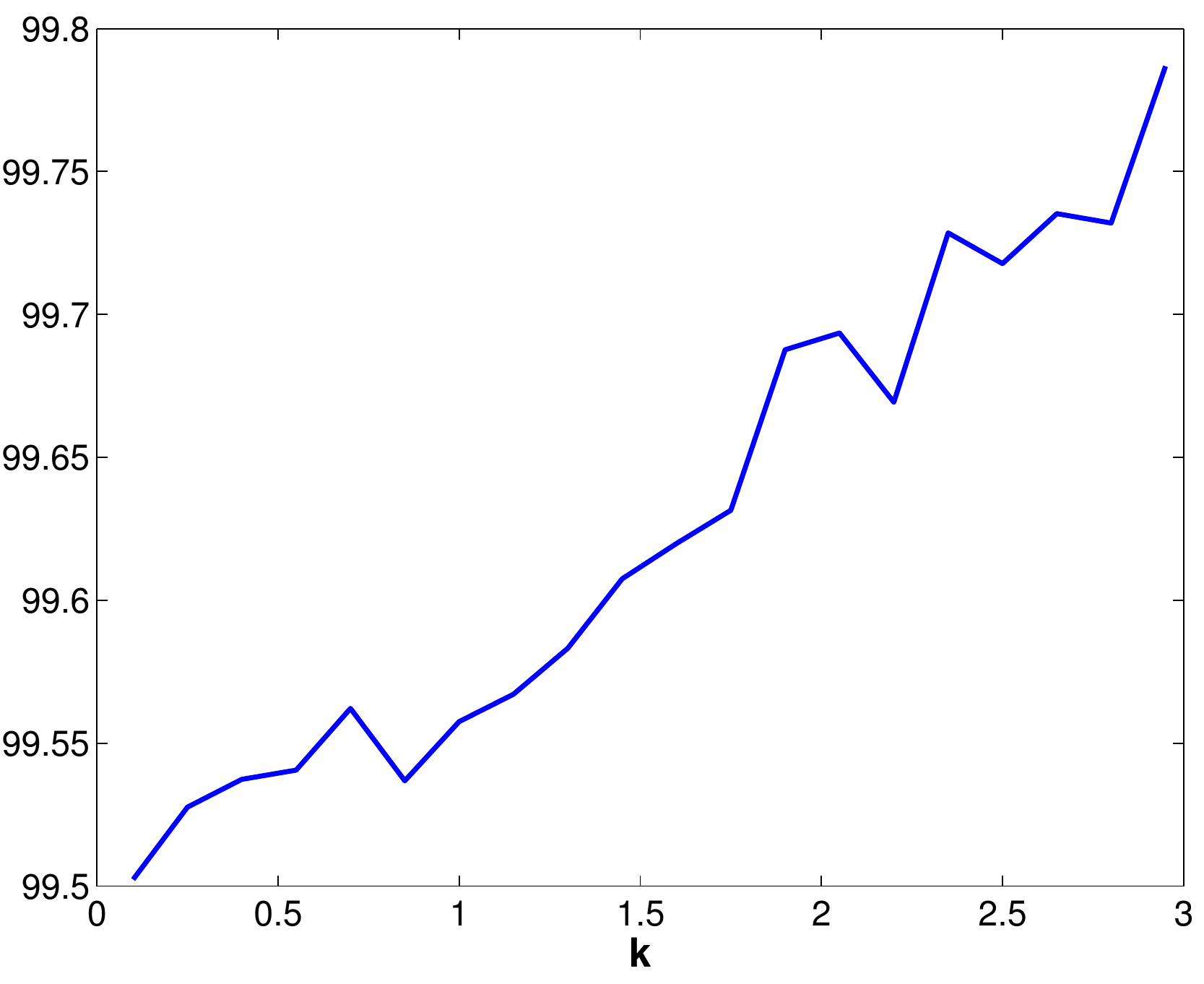}
\caption{Average mean-square error of $\dot{r}$}
\label{fig:noisy-rdot_erravg2}
\end{figure}

\subsection{Measurement Error with Constant Wind}
Here we examine the performance of the algorithm under the influence
of measurement errors (as above) and constant wind.
One may formulate the `windy' system with
constant wind bias of speed $W_s$ and direction $w_d$ as
\begin{align}
  \begin{aligned} \label{wind-form}
    d \left[
    \begin{array}{c}
      x \\ y \\ \psi
    \end{array}
  \right]
  = \left[
  \begin{array}{c}
    V\cos\psi +W_s \cos w_d \\
    V\sin\psi +W_s \sin w_d \\
    u
  \end{array}
\right] dt 
  +\left[ 
  \begin{array}{c}
  0 \\ 0 \\ -k\sigma
  \end{array}
  \right] d \xi .
  \end{aligned}
\end{align}
We simulate trajectories under such a wind model, using the same control policy $u$ as in \eqref{u-defn}. We take the windspeed $W_s=V/4=0.25$
and the wind direction $w_d=\pi/4$.
\begin{figure}[ht]
\begin{minipage}{0.45\linewidth}
\centering
\includegraphics[width=\textwidth]{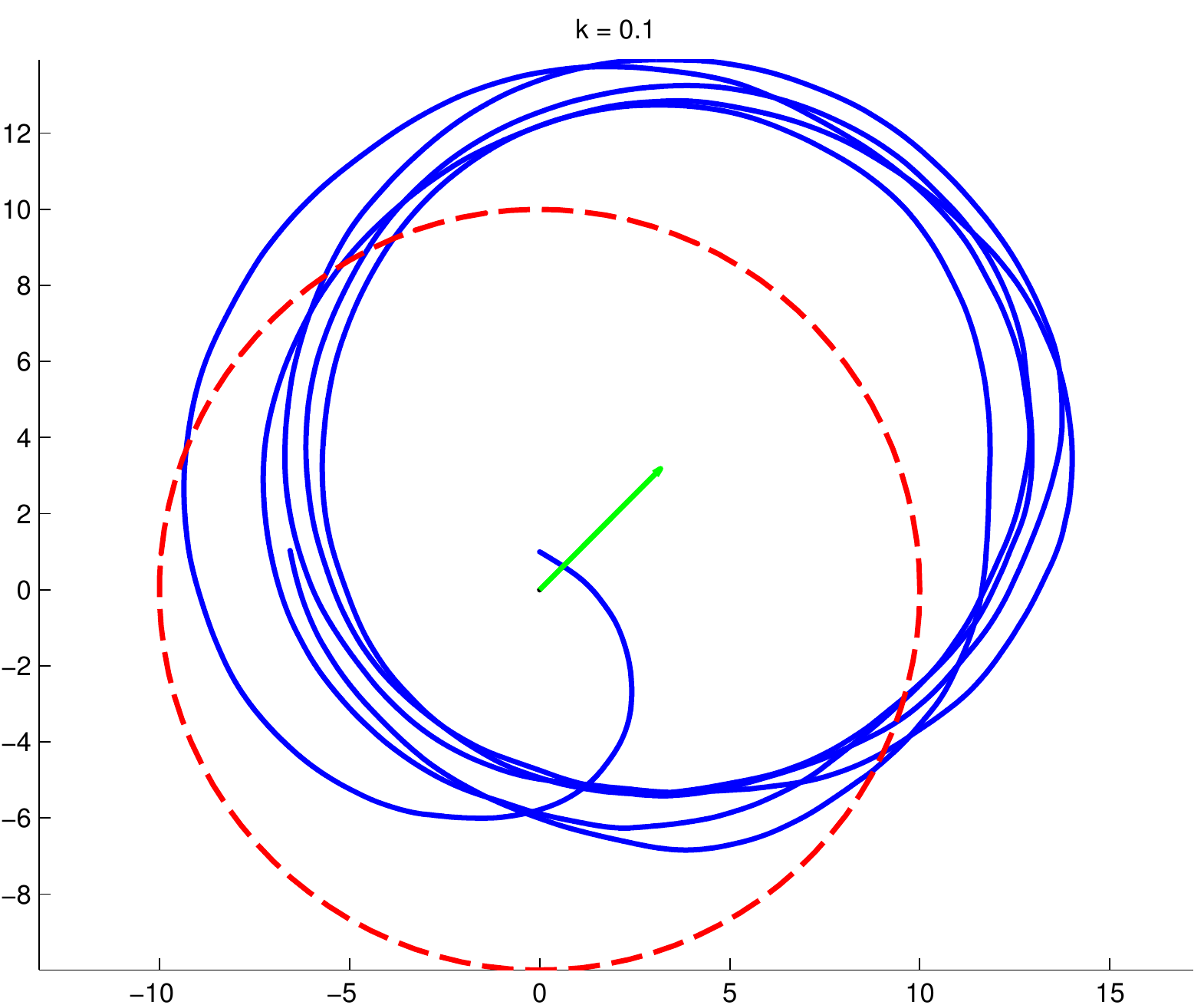}
\caption{Trajectory with measurement error and wind, $k=0.1$}
\label{fig:windy-k_0_1}
\end{minipage}
\begin{minipage}{0.45\linewidth}
\centering
\includegraphics[width=\textwidth]{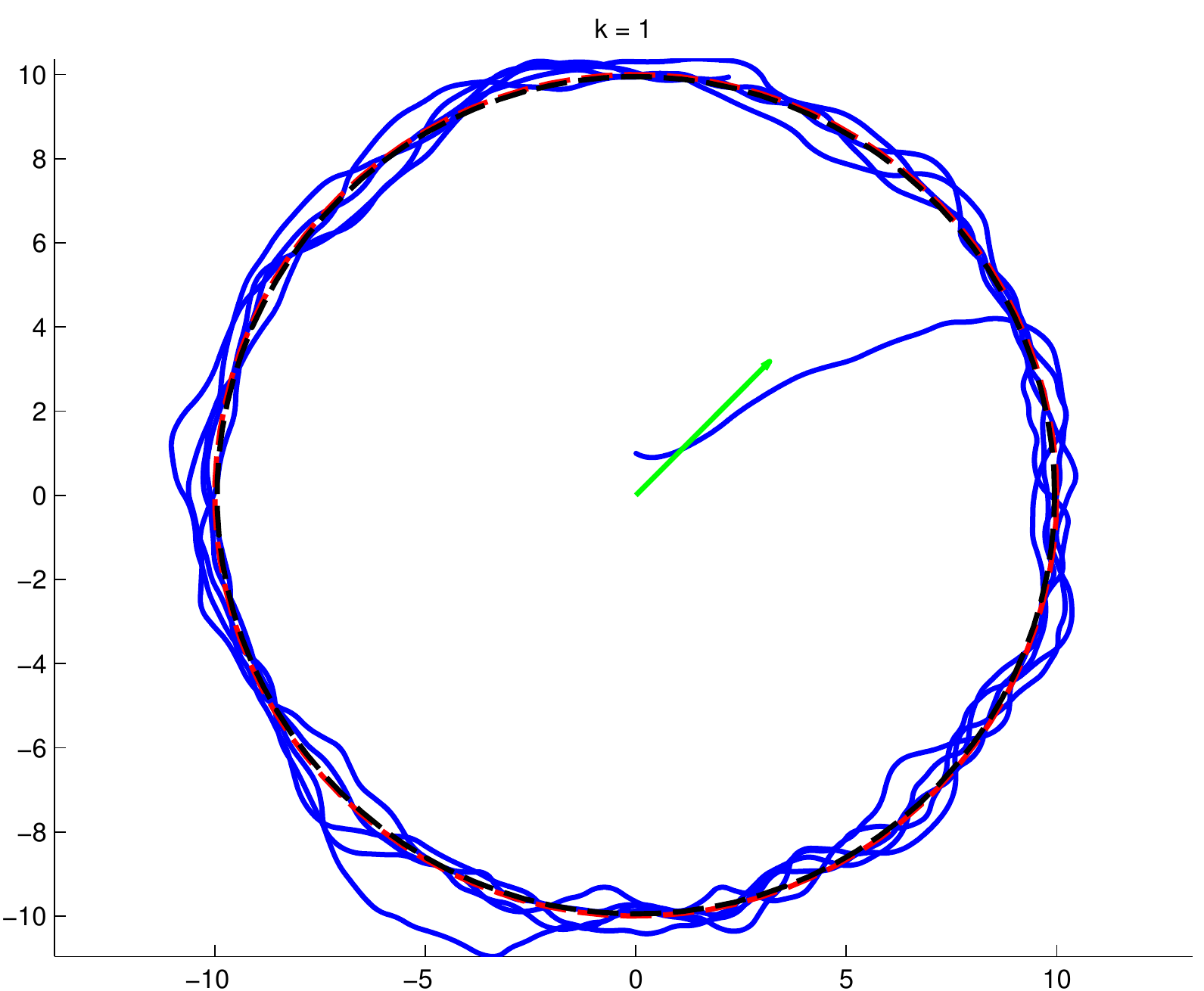}
\caption{Trajectory with measurement error and wind, $k=1$}
\label{fig:windy-k_1_0}
\end{minipage}
\end{figure}
Figures~\ref{fig:windy-k_0_1} and \ref{fig:windy-k_1_0} depict the
windy trajectories analogous to the windy case. We note that with
the minimal gain size the trajectory forms a circular orbit, but is shifted
off-target in the direction of the wind. When the gain is turned up the UAV
adjusts more dynamically and is able to adhere to the desired radius much better.
\begin{figure}[h]
\centering
\includegraphics[width=0.75\linewidth]{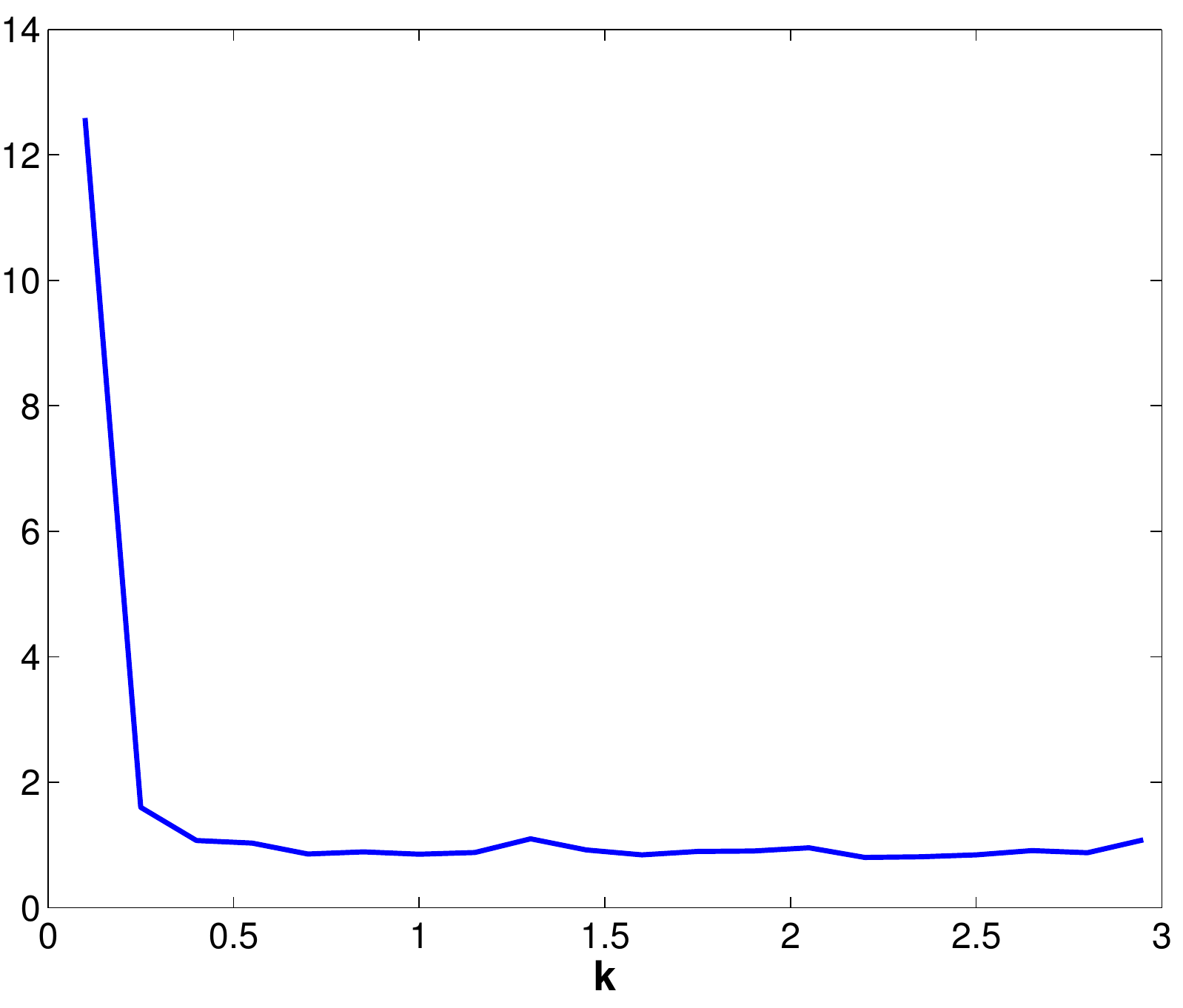}
\caption{Average mean-square error of $(r-r_a)$ with wind}
\label{fig:windy-r_erravg2}
\end{figure}
\begin{figure}[h]
\centering
\includegraphics[width=0.75\linewidth]{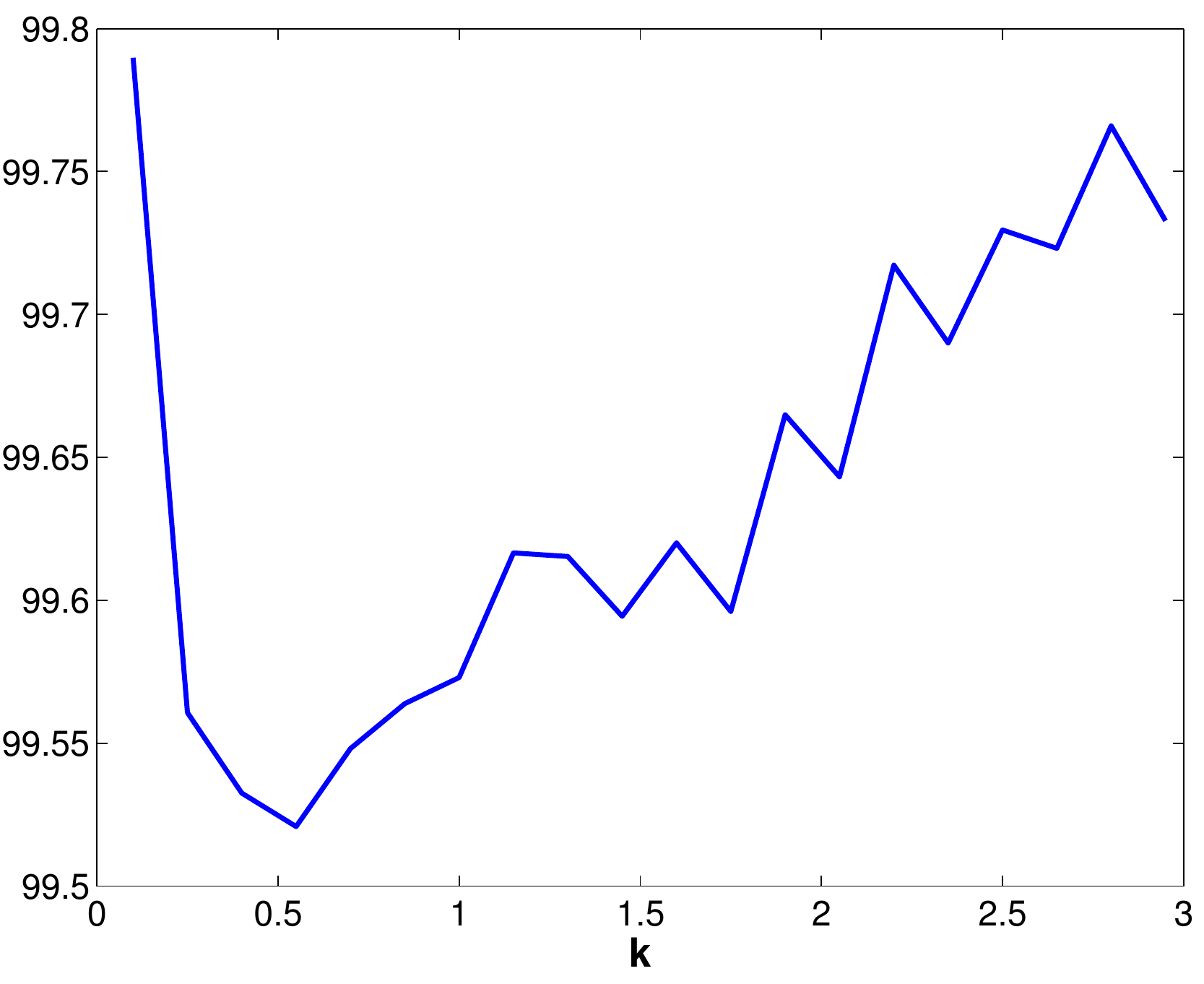}
\caption{Average mean-square error of $\dot{r}$ with wind}
\label{fig:windy-rdot_erravg2}
\end{figure}
Figures~\ref{fig:windy-r_erravg2} and \ref{fig:windy-rdot_erravg2} show the
mean-square error of $(r-r_a)$ and $\dot{r}$ under the influence of wind
and measurement errors.

\section{Conclusion and Future Work} \label{sec:conclusion}
This paper has established a robust control policy for a UAV to
circumnavigate a stationary target using noise-corrupted range and range rate
measurements,
without any use or assumption of location information for the UAV nor
the target. Assuming additive measurement errors we established a
 recurrence result, bounding the time until the UAV reaches
a neighborhood of the desired orbit, via a Lyapunov
function approach. A simulation study was then used to collect statistics
of the performance of the control policy with measurement errors, as well
as with drifting bias due to the influence of wind.

Future work may attempt to establish that the trajectory is set-wise
stable to the recurrent set, as simulations seem to suggest. Traditional
stochastic stability results as in \cite{Khasminskii_Stochastic_2011} are
not applicable due to the persistence of noise (non-zero diffusion
coefficient) at the `stability' point $(r_d,\pi/2)$. However, $p${th}-moment
set-wise stability in the sense of \cite{Nunez_Stability_2013} may be
possible.


Other research directions include formal analysis of the system with
constant wind bias as in \eqref{wind-form}.
The addition of wind terms in $\dot{x}, \dot{y}$ prevent the reduction of
the system to $(r,\theta)$. However, assuming one can additionally
measure the heading angle $\psi$ (by addition of a magnometer), it is
possible to formulate the current control and windy system dynamics in
terms of $(r,\theta,\psi)$. Such conversion assumes $W_s$ and $w_d$
are known, but it may be possible to statistically estimate these
quantities from a few revolutions of the target under the current control.
For example, one sees in Figure~\ref{fig:windy-k_0_1} that
with small gain there is significant bias of the orbit in direction of the
wind. One may attempt to first estimate the wind direction $w_d$ as a
statistical change-point problem from when the radius is under-biased to
when it is over-biased. One may then try to estimate wind speed $W_s$
by the magnitude of such a change.

Finally, the addition of heading angle $\psi$ measurements may allow for other control
schemes to be developed,  perhaps resulting smoother trajectories and
less control effort.


\begin{thebibliography}{99}

\bibitem{Gertler_US_2012}
J. Gerler, "U.S. Unmanned Aerial Systems",
Congressional Research Service Report, Jan. 2012. [Online].
Available: http://www.fas.org/sgp/crs/natsec/R42136.pdf

\bibitem{Warwick_Lightsquared_2011}
G. Warwick, �Lightsquared tests confirm GPS jamming,�
Aviation Week,  June 2011. [Online].	
Available: http://www.aviationweek.com/aw/generic/story.jsp?id=news/awx/ 2011/06/09/awx06092011p0- 334122.xml

\bibitem{Shepard_Dronehack_2012}
D. Shepard, J. Bhatti, and T. Humphreys, �Drone hack: Spoofing attack demonstration on a civilian unmanned aerial vehicle,� GPS World, 2012. [Online]. Available: http://www.gpsworld.com/drone-hack/

\bibitem{Shames_Circumnavigation_2012}
I. Shames, S. Dasgupta, B. Fidan, and B. D. O. Anderson, �Circumnavigation Using Distance Measurements Under Slow Drift,� IEEE Transactions on Automatic Control, vol. 57, no. 4, pp. 889�903, 2012.

\bibitem{Deghat_Target_2013}
M. Deghat, I. Shames, B. D. O. Anderson, and C. Yu, �Target localization and circumnavigation using bearing measurements in 2D,� in IEEE Transactions on Automatic Control, 2013.

\bibitem{Cao_Circumnavigation_2013}
Y. Cao, J. Muse, D. Casbeer, and D. Kingston,
"Circumnavigation of an Unknown Target Using UAVs with Range and Range Rate Measurements", to appear in IEEE Conference on Decision and Control, 2013, available at arxiv.org/abs/1308.6250.

\bibitem{Khasminskii_Stochastic_2011}
R. Khasminskii, {\it Stochastic Stability of Differential Equations}, Springer-Verlag, Berlin, 2011.

\bibitem{Oksendal_Stochastic_2003}
B. K. {\O}ksendal, {\it Stochastic Differential Equations: An Introduction with Applications}, Springer-Verlag, Berlin, 2003.

\bibitem{Nunez_Stability_2013}
D. Mateos-N\'un\~ez and J. Cort\'es, �Stability of stochastic differential equations with additive persistent noise,� in American Control Conference (ACC), 2013, 2013, pp. 5427�5432.



\end{thebibliography}
\end{document}